\documentclass{amsart}

\usepackage{amsmath, amssymb, amsbsy}
\usepackage{array}
\usepackage{graphpap, color, paralist,xcolor}
\usepackage[mathscr]{eucal}
\usepackage[pdftex]{graphicx}
\usepackage[pdftex,colorlinks,backref=page,citecolor=blue]{hyperref}

\usepackage{mathtools}
\usepackage{cleveref}
\usepackage{float}
\usepackage{crossreftools}

\newcounter{bullet}

\usepackage{setspace}
\setdisplayskipstretch{1}

\usepackage{geometry}
\usepackage{comment}
\newtheorem{thm}{Theorem}[section]
\newtheorem{prop}[thm]{Proposition}
\newtheorem{cor}[thm]{Corollary}
\newtheorem{lem}[thm]{Lemma}

\newtheorem{conj}[thm]{Conjecture}

\theoremstyle{definition}
\newtheorem{mydef}[thm]{Definition}
\newtheorem{claim}[thm]{Claim}

\crefname{lem}{lemma}{lemmas}

\newcommand{\gO}{\Omega}

\newcommand{\EE}{\mathbb{E}}

\newcommand{\beq}[1]{\begin{equation}\label{#1}}
\newcommand{\enq}[0]{\end{equation}}

\newcommand{\ga}[0]{\alpha }
\newcommand{\gb}[0]{\beta }
\newcommand{\gc}[0]{\gamma }
\newcommand{\gd}[0]{\delta }
\newcommand{\gD}[0]{\Delta }

\newcommand{\vp}[0]{\varphi}

\newcommand{\bn}[0]{\bigskip\noindent}
\newcommand{\mn}[0]{\medskip\noindent}
\newcommand{\nin}[0]{\noindent}

\newcommand{\ov}[0]{\bar}
\newcommand{\sub}[0]{\subseteq}
\newcommand{\sm}[0]{\setminus}
\newcommand{\0}[0]{\emptyset}
\newcommand{\ra}[0]{\rightarrow}

\newcommand{\pr}[0]{\mathbb{P}}

\newcommand{\Aut}[0]{\mbox{Aut}}
\newcommand{\aut}[0]{\mbox{aut}}

\newcommand{\pE}[0]{p_{\mathbb E}}

\newcommand{\cop}[1]{\ov{#1}}

\newcommand{\C}[2]{\binom{{#1}}{{#2}}}

\newcommand{\T}[0]{{\mathcal T}}
\newcommand{\sss}[0]{{\mathcal S}}

\newcommand{\Tjs}[0]{T_j^*}
\newcommand{\Sjs}[0]{S_j^*}
\newcommand{\vjs}[0]{v_j^*}
\newcommand{\ejs}[0]{e_j^*}
\newcommand{\ST}[0]{R}

\newcommand{\CC}{C}
\newcommand{\rr}{r}

\newenvironment{subproof}[1][\proofname]{
  
  \begin{proof}[#1]
}{
  \end{proof}
}

\begin{document}

\title{On the ``second" Kahn--Kalai Conjecture}

\author[Q. Dubroff]{Quentin Dubroff}
\address{Department of Mathematics, Carnegie Mellon University}
\email{qdubroff@andrew.cmu.edu}

\author[J. Kahn]{Jeff Kahn}
\address{Department of Mathematics, Rutgers University}
\email{jkahn@math.rutgers.edu}

\author[J. Park]{Jinyoung Park}
\address{Department of Mathematics, Courant Institute of Mathematical Sciences, New York University}
\email{jinyoungpark@nyu.edu}

\begin{abstract}
We make progress on a conjecture of Kahn and Kalai, 
the original (stronger but less general) version of what became known as the ``Kahn-Kalai Conjecture" (KKC; now 
a theorem of Park and Pham). 
This ``second" KKC concerns the threshold, $p_c(H)$, for  
$G_{n,p}$ 
to contain a copy of a given graph $H$, predicting
$p_c(H) = O(\pE(H)\log n)$, where $\pE$ is an easy lower bound on $p_c$.
What we actually show is $\pE^*(H)=O(\pE(H)\log ^2n)$, where 
$\pE^*$,
the \emph{fractional expectation threshold}, 
is a larger lower bound suggested by Talagrand.  
When combined with Talagrand's fractional relaxation of the KKC 
(now a theorem of Frankston, Kahn, Narayanan and Park),
this gives $p_c(H)=O(\pE(H)\log^3 n)$.
(The second KKC would follow similarly if one could
remove the log factors from the above bound on $\pE^*$.)
\end{abstract}

\maketitle
%%%%%%%%%%%%%%%%%%%%%%%%%%%%%%%%%%%

\section{Introduction}\label{sec.intro}

For graphs $G$ and $H$, a \emph{copy} of $H$ in $G$ is an (unlabeled)
subgraph of $G$ isomorphic to $H$.
We use $N(G,H)$ for the number of such copies.
As usual, $G_{n,p}$ is the Erd\H{o}s-R\'enyi random graph; that is,
it has vertex set $[n]:=\{1,\ldots, n\}$, and contains each edge of
$K_n$ with probability $p$, independent of other choices.
(We will, a little abusively, use ``$G\supseteq H$"
to mean $G$ contains a copy of $H$.)
Additional notation is given at the end of this section.

Define the \textit{threshold for $H$-containment}, 
$p_c(H)$ ($= p_c(n,H)$, but we suppress the $n$),
to be the unique $p$ for which 
$\pr(G_{n,p}\supseteq H)=1/2$.
This is finer than the original notion
of Erd\H{o}s and R\'enyi \cite{erdos1960evolution}, 
which defines $p_0=p_0(n)$ to be 
\emph{a threshold function} for the property $\{G_{n,p}\supseteq H\}$ if 
\[
\pr(G_{n,p} \supseteq H) \rightarrow
\left\{\begin{array}{ll}
$0$&\mbox{if $p \ll p_0$}\\
$1$& \mbox{if $p \gg p_0$}
\end{array}\right\}
\,\,\mbox{as $n \rightarrow \infty$.}
\]
That $p_c(H)$ is always such a threshold function follows from \cite{bollobas1987threshold}.

For a graph $J$, set
$\EE_pX_J =\EE N(G_{n,p},J).$
In \cite{kahn2007thresholds}, 
Kalai and the second author considered
a naive lower bound on $p_c(H)$ that one might now
call the 
\textit{graphic expectation threshold for $H$}:
\[
\pE(H) =\pE(n,H)=\min\{p:\EE_pX_I \ge 1/2 ~ \forall I \sub H\}.
\]
(This is a lower bound on
$p_c(H)$ since,
for any $I\sub H$,
$\pr(G_{n,p}\supseteq H) \leq \pr(G_{n,p}\supseteq I) \leq \EE_pX_I$.
Strictly speaking, \cite{kahn2007thresholds} defines $\pE$ with
1 in place of 1/2.  
It also calls $\pE$ the 
\textit{expectation threshold},  
and gives no name to what has 
been called the ``expectation threshold"
since \cite{frankston2021thresholds}.)

Our focus in the present work is on the following suggestion from
\cite{kahn2007thresholds} to the effect that 
$p_c(H)$ and $\pE(H)$ are never very far apart;
this conjecture, called the ``second Kahn--Kalai Conjecture" in 
\cite{mossel2022second}, was in fact the starting point for 
\cite{kahn2007thresholds}.  
(We use $v_H$ and $e_H$ for $|V(H)|$ and $|E(H)|$.)

\begin{conj}\cite[Conjecture 2.1]{kahn2007thresholds}\label{KKC2}
There is a fixed $K$ such that for any graph $H$, 
\[p_c(H) < K \pE(H) \log v_H.\]
\end{conj}
\nin
In the limited graphic setting to which it applies, 
Conjecture~\ref{KKC2} is considerably stronger than the main 
conjecture of \cite{kahn2007thresholds} 
(called the ``Kahn--Kalai Conjecture" since
\cite{talagrand2010many}),
which was recently proved by Pham and the third author \cite{park2024proof}. 
Here we show that this stronger version isn't \emph{too} far from 
correct, in that it is never off by more than a $\log^2n$ factor,
and that it does hold in a sufficiently sparse regime:

\begin{thm}\label{MT'} There is a fixed $K$ such that for any graph $H$,
\[p_c(H) \le K \pE(H) \log v_H \log ^2 n.\]
\end{thm}

\begin{thm}\label{MT'_sparse} 
There is a fixed $K$ such that if 
$\pE(H)<1/(3n)$, then 
\[
p_c(H) < K \pE(H)\log v_H.
\]
\end{thm}
\nin
(While it's hard to judge past efforts on \Cref{KKC2}
absent published progress,
a comment from \cite{KLLM} is suggestive:
``but this [conjecture] is widely open, even if $\log n$ is replaced by $n^{o(1)}$.")

\mn

What we actually show, 
strengthening Theorems~\ref{MT'} and \ref{MT'_sparse},
is a connection between $\pE$ and another 
lower bound, the \emph{fractional expectation threshold for H}, which 
in our context may be defined as
\beq{def.pe} 
\pE^*(H)=\min\{p:\EE_pX_I \ge N(H,I)/2 \,\,\,\, \forall I \sub H\}.
\enq
(Note ``$\sub H$" is superfluous.)
In this case, 
$\pE^*(H)\leq p_c(H)$ is given by Markov's Inequality:
for any $I$, 
$\pr(G_{n,p}\supseteq H) \leq \pr(N(G_{n,p},I) \geq N(H,I)) \leq \EE_pX_I/N(H,I)$.

We omit the usual definition of
fractional expectation threshold, due to 
Talagrand \cite{talagrand1995all, talagrand2010many} 
(but named in \cite{frankston2021thresholds}), 
which applies to general increasing properties.  
It is not hard to see that in the present
context it agrees with the definition in \eqref{def.pe}
(briefly:  the fractional expectation threshold
is the maximum of a collection of lower bounds on $p_c$, which
maximum is, in the case of
graph containment, attained by 
the bound corresponding to one of the constraints in 
\eqref{def.pe}; see also \cite{mossel2022second}).

\mn

Trivially, 
\[\pE(H) \le \pE^*(H),\]
and it was shown in
\cite{frankston2021thresholds} that \Cref{KKC2} holds with $\pE$ replaced by $\pE^*$:   

\begin{thm}\cite[Theorem 1.1]{frankston2021thresholds}
\label{FKNP} There is a fixed $K$ such that for any graph $H$,
\[p_c(H) < K\pE^*(H)\log e_H.\]    
\end{thm}
\nin
(Here we could replace $e_H$ by $v_H$ to agree with 
\Cref{KKC2}.  The present version, which
corresponds to what's in \cite{frankston2021thresholds}, is formally 
stronger since $\log e_H\leq 2\log v_H$,
but easily seen to be equivalent.)

\mn

In view of Theorem~\ref{FKNP}, 
Theorems~\ref{MT'} and \ref{MT'_sparse} are implied
by the following assertion.
The first part of this should be 
considered the main point of the present work, but we would also
like to 
stress the second part, as a small success with
Conjecture~\ref{CpEpE*} below, which is our ultimate goal (see also the 
remark following the conjecture).

\begin{thm}\label{MT''}
There is a fixed $K$ such that for any graph $H$,
\[
\pE^*(H) < K \pE(H)\log^2 n.
\]
Furthermore, there is a fixed $\ga>0$ such that if 
$\pE(H)<1/(3n)$, then $\pE^*(H) < \ga^{-1}\pE(H)$.
\end{thm}
\nin 
Our present belief is that the extra log factors are never necessary; that is, that the following 
strengthening of 
Conjecture~\ref{KKC2} holds.
\begin{conj}\label{CpEpE*}
There is a fixed $K$ such that for any graph $H$,
$\,\,\pE^*(H) < K \pE(H)$.
\end{conj}
\nin
(This is stronger than a positive answer to 
\cite[Question 2.2]{kahn2007thresholds}, which asks whether
the ``expectation threshold" 
of \cite{frankston2021thresholds},
which is never more than $\pE^*(H)$, is $O(\pE(H))$.)

\mn

For the proof of Theorem~\ref{MT''}, 
we work with a reformulation.
Given $q \in [0,1]$, say a graph $J$ is \textit{$q$-sparse}
if 
\[
\EE_qX_{I} \ge 1 \,\,\,\forall I\sub J,
\]
and note that, apart from the irrelevant switch from 1/2 to 1
(which we abusively ignore),
this is just another way of saying $\pE(H) \le q$.
We then have the following appealing
restatement of Conjecture~\ref{CpEpE*}, and the promised
reformulation of Theorem~\ref{MT''}.
\begin{conj}\label{CpEpE*'}
There is a fixed $K$ such that
if $H$ is $q$-sparse and $p=Kq$, then
\[
N(H,F) < \EE_pX_F \quad \forall F \sub H.
\]
\end{conj}
\nin
(Again ``$\sub H$" is unnecessary.) 

\mn
\emph{Remark.}
Conjecture~\ref{CpEpE*'} is challenging even 
for simple classes of $F$'s.
In a separate paper \cite{DKP3} we 
prove it for a few such classes, namely
(a) cliques, (b) bounded degree forests, and (c) cycles, 
which turn out to be (a) not hard (in retrospect), 
(b) quite hard (as in, \emph{hard to find},
though some of the details are also not easy), and
(c) rather hard, even given the result for paths 
in (b).
Ideas developed to prove these results played an important role 
in the discovery of the proof of Theorem~\ref{MT}.

\begin{thm}\label{MT}
There is a fixed $K$ such that if $H$ is $q$-sparse and 
$p=Kq\log^2 n$, then
\beq{eq.MT} 
N(H,F) < \EE_pX_F \quad \forall F \sub H.
\enq
Furthermore, there is a fixed $\ga>0$ such that
\eqref{eq.MT} holds if $H$ is $q$-sparse with
$q=\ga p<1/(3n)$.
\end{thm}
\nin(Eventually $\ga\approx e^{-2}$ is enough;
see the end of Section~\ref{sec.MT.sparse}.)

\mn
\emph{Remark.}
The proof of \Cref{MT} also shows that if $|E(F)|<C $, then
\eqref{eq.MT} holds for $p=K_{_C} q\log n$. 

\mn

The rest of the paper is organized as follows. 
\Cref{sec.prelim} 
includes definitions and a few easy points, and
sketches the main line of argument for
\Cref{MT}. 
This is implemented in
Sections \ref{sec.MT} and \ref{sec.lem.tree.hard}, and 
the quite different argument for smaller $q$ is given in \Cref{sec.MT.sparse}.

\subsection*{Usage}

For a graph $J$ and $U\sub V(J)$, the neighborhood,
$N(U)=N_J(U)$, is the set of vertices with neighbors in $U$, and
$J[U]$ is the subgraph of $J$ induced by $U$.  
As above, the sizes of 
$V(J)$ and $E(J)$ are denoted $v_J$ and $e_J$ (or $v(J)$ and $e(J)$). 
We use $[a]=\{1,2,\ldots,a\}$ and $(n)_a=n(n-1)\cdots (n-a+1)$
(for positive integer $a$), and,
throughout the paper, take $\log$ to mean $\log_2$. 
We make no effort to optimize constant factors,
and, in line with common practice, 
pretend all large numbers are integers.

\section{Proof preview and preliminaries}\label{sec.prelim}

Here we introduce a few basic notions that 
provide the framework for most of what follows.
With these in hand, we then
sketch a portion of the proof of
\Cref{MT}, 
both to give a general idea of strategy and
to point to sources of the 
``extra'' $\log$ factors that 
\Cref{CpEpE*'} says should be unnecessary.
In closing the section we note some easy points, 
in particular
slightly simplifying our assignment by showing that 
in proving Theorem~\ref{MT} it is enough to consider 
connected $F$.

\subsection*{Extensions and leading decomposition}

For $W \subsetneq Z \sub V(J)$, the \emph{subgraph of $J$ on $Z$ rooted at $W$}, denoted $J[W,Z]$, has vertex set $Z$
and edge set $E(J[Z]) \setminus E(J[W])$.
When the identity of $J$ is either clear or 
irrelevant (e.g.\ for the rest of the present discussion), we
usually abbreviate
$J[W,Z]=[W,Z]$.
(Any $J$ used without qualification is a general graph.)
We use $e(W,Z)$ for the number of edges of $[W,Z]$ and $v(W,Z)$ for $|Z~\setminus~W|$.
We will also use $[B,A]$ for $[B,B\cup A]$,
and similarly for $v(B,A)$ and $e(B,A)$
(note this includes edges contained in $A$).
The set of 
automorphisms of $J$ is
$\Aut(J)$, and
$\Aut(W,Z)$ is the set of automorphisms of $[W,Z]$ 
that fix $W$ pointwise.
We will always use $\aut(\cdot)$ for $|\Aut(\cdot)|$.

In what follows we will have: $J$,
$W,Z$ as above; 
a graph $G$ (usually $H$) and $U\sub V(G)$ of size $|W|$;
and orderings $(w_1, \ldots, w_s)$ and $(u_1,\ldots,u_s)$ of $W$ and $U$, which will be
left implicit.
A \textit{$[W,Z]$-extension on $U$} (\textit{in $G$} if the 
identity of $G$ is not clear) is then an injection
$\psi:Z\ra V(G)$ that extends the map $w_i \mapsto u_i$ ($i \in [s]$)
and is \emph{edge-preserving} (i.e.\
$\{x,x'\}\in E(J)\, \Rightarrow \,\{\psi(x), \psi(x')\}\in E(G)$). When $W=\0$, such an extension is a 
\emph{labeled copy} of $J[Z]$ in $G$.
Rooted graphs $[W,Z]$ and $[U,V]$ with implicit orderings $(w_1, \ldots, w_s)$ and $(u_1,\ldots,u_s)$ of $W$ and $U$ are 
\emph{isomorphic} if there is an isomorphism of $Z$ to $V$
extending the map $w_i \mapsto u_i$ ($i \in [s]$).

Extensions $\psi$ and $\psi'$ 
are \textit{equivalent} if $\psi=\psi' \circ \gamma$ for some $\gamma \in \Aut(W,Z)$, and 
a \textit{copy of $[W,Z]$ on $U$} is an equivalence class of such
extensions (so an ``unlabeled" extension, but note we continue to label $W$ and $U$). 

\mn

We use $\tilde \tau_{[W,Z]}(U,G)$ and $\tau_{[W,Z]}(U,G)$ 
for the numbers of $[W,Z]$-extensions and copies of $[W,Z]$ on $U$ in $G$ 
(so $\tilde \tau_{[W,Z]}(U,G)= \tau_{[W,Z]}(U,G)\cdot \aut(W,Z)$), and set
\beq{mumu}
\tilde \mu_p(W,Z)=n^{v(W,Z)}p^{e(W,Z)} \text{ and } \mu_p(W,Z)=\tilde \mu_p(W,Z)/\aut(W,Z);
\enq
these are surrogates for the uglier 
$\EE ~\tilde \tau_{[W,Z]}(U,G)$ 
($= (n-|W|)_{v(W,Z)}p^{e(W,Z)}$) and $\EE ~\tau_{[W,Z]}(U,G)$.
We also use $\tilde \mu_p(J[W,Z])$ and $\mu_p(J[W,Z])$ 
if $J$ is unclear, and 
$\mu_p(B,A)$ for $\mu_p(B,B\cup A)$ (and similarly
for $\tilde \mu_p$).

\mn

The next notion is key. 
\begin{mydef}\label{def.leading} 
    A rooted graph $[W,Z]$ is \textit{$p$-leading} if $\mu_p(Y,Z)<1$ for all $Y$ satisfying $W \subsetneq Y \subsetneq Z$. 
\end{mydef}

\begin{prop}\label{lem.leading seq} 
If $0 \le q \le p \le 1$ 
and $J$ is $q$-sparse, 
then there is a sequence
\[
\emptyset=W_0 \subsetneq W_1 \subsetneq \cdots \subsetneq W_k=V(J)
\]
such that $J[W_{i-1}, W_i]$ is $p$-leading and $\mu_p(W_{i-1}, W_i) \ge 1$ 
for $i \in [k].$ 
\end{prop}

\begin{proof} 
Note that if $J$ is $q$-sparse, then
(i) $J$ is $q'$-sparse for every $q'>q$,
and (ii) every subgraph of $J$ is $q$-sparse.    
By (i), we just need to prove the proposition when $p=q$.  
We proceed by induction on $v_J\geq 1$
(with base case vacuous).
Since $q$-sparseness gives
$\mu_q(\emptyset, V(J))\ge \EE_qX_J \ge 1$, we may
choose $Y \subsetneq V(J)$ maximal with $\mu_q(Y, V(J)) \ge 1$. 
By maximality, $[Y,V(J)]$ is $q$-leading; so 
either $Y=\emptyset$ and we are done, 
or, since $J[Y]$ is $q$-sparse (by (ii)),
induction gives a sequence 
$\emptyset=W_0 \subsetneq W_1 \subsetneq \cdots \subsetneq W_{k-1}=Y$ as in the proposition, and appending 
$W_k:= V(J)$ completes the proof.
\end{proof}

\mn

\subsection*{Sketch}

Here we try to give some idea of what will 
happen in the main line of argument for
Theorem~\ref{MT}.
For this informal discussion we call labeled copies
\emph{embeddings}.
We fix a 
decomposition
\[
\0 = W_0\subsetneq W_1 \subsetneq \dots \subsetneq W_k = V(F)
\]
as in \Cref{lem.leading seq},
with $p$ replaced by $q' \approx q \log n$,
and, for each $i\in [k]$, bound 
the number of ways an embedding 
of $F[W_{i-1}]$ in $H$ can be extended to an embedding 
of $F[W_i]$. 
(One of the extra log factors in Theorem~\ref{MT} can be blamed
on our use of 
this worst case bound for
quantities that should usually be much smaller.)

For the bound,
we note that the extensions of a
given embedding divide into 
``components," and show, using the sparseness of $H$,
that (i) the number of  
components is not too large, and (ii)
no component contains too many extensions.
These assertions are shown in Lemmas~\ref{F.fixed.lem1} and
\ref{F.gen.lem}, with the first easy and 
the second the main part of our argument.

The main point for \Cref{F.gen.lem}
is \Cref{lem.tree.hard}, which says, roughly,
that a union, $T$,
of many overlapping extensions 
of a given $F[W_{i-1}]$ has
small $\EE_qX_T$, so cannot appear
in a suitably sparse $H$. 
Crucial here is limiting the numbers
of automorphisms of our extensions:
the contribution to our bound on
$\EE_qX_T$ (which is really a bound on the 
number of \emph{labeled} copies)
of an extension of 
``type" $[Y,Z]$ with $Y\neq W_{i-1}$ is essentially
$\tilde\mu_q(Y,Z) = \mu_q(Y,Z)\cdot\aut(Y,Z)$,
and $\mu_q(Y,Z)$ is small because 
$[W,Z]$ is leading; so the $\aut(Y,Z)$'s become the issue.
(Note that this gets us only so far, since 
such ``local" symmetries may not be reflected in
$\aut(T)$; see following \Cref{lem.tree.hard}
for a possible improvement that 
would eliminate the second ``extra" $\log$
in Theorem~\ref{MT}.)

Control of automorphism counts is mainly based on
Lemmas~\ref{lem.no sunflower}~and~\ref{lem.sunflower},
the more important of which, Lemma~\ref{lem.sunflower}, 
says that a graph with many automorphisms must contain 
a large ``sunflower'' (see \Cref{def.sunflower}). 
Concretely, we show that for $k >4\log n$, a graph 
(or rooted graph) $J$ with $n$ vertices and 
no $k$-sunflower has at most $O(k)^{e_J}$ automorphisms. We 
conjecture this holds for \emph{all} $k$, 
which seems to us both interesting in itself
and a probably necessary step toward the above-mentioned
improvement of \Cref{lem.tree.hard}.

\mn

\subsection*{Small points}
We now turn to the above-mentioned observations 
and reduction.

\begin{prop}\label{prop.small claims}
If 
$J[W,Z]$ is $p$-leading, with $Z \setminus W=A$ and $N(A) \cap W=B$, then:

\mn
{\rm (a)}
$\mu_p(W,Z)\le n$, with the inequality strict if $|A| \ge 2$;

\mn
{\rm (b)}
if $p\ge 1/n$, then either $|A|=1$ or each $v \in A$ has degree at least $2$ in $[B,A]$;

\mn
{\rm (c)}
if $p\ge 1/n$, $|A|\ge 2$, and $\mu_p(W,Z)\ge 1$,
then $e(B,A)\ge \max\{|A|,|B|+1\} $, and  
$e(B,A) > |A|$ unless $B=\0$ and $J[A]$ is a cycle.

\end{prop}

\begin{proof}
    (a) This is trivial if $|A|=1$; otherwise, for any $v \in A$,
    \[\mu_p(W,Z)=\mu_p(W \cup \{v\},Z)np^{e(W,W\cup \{v\})}\frac{\aut(W \cup \{v\},Z)}{\aut(W,Z)}\le \mu_p(W \cup \{v\},Z) n<n,\]
where the equality is two applications of 
the second part of \eqref{mumu}, and
the final inequality holds since $[W,Z]$ is $p$-leading.

\mn
(b) If $|A|>1$ and $v \in A$ has degree at most $1$ in $[B,A]$, then
$\mu_p(Z\sm \{v\},Z)\geq np\ge 1$ contradicts the 
assumption that $[W,Z]$ is $p$-leading.

\mn
(c)
This follows from (b) 
and the observation that if $[W,Z]$ is $p$-leading with
$\mu_p(W,Z)\ge 1$, then $J[Z\sm W]$ is connected
(since otherwise, if $X_1\dots X_k$ are the vertex sets of 
components of $J[Z\sm W]$, then
$1\le \mu_p(W,Z)\le\prod \mu_p(Z\sm X_i,Z)<1$). 
\end{proof}

\mn

\subsection*{Reduction}
Our last job in this section is to show that in proving Theorem~\ref{MT},
we need only consider connected $F$.
We recall a familiar, elementary estimate:
\beq{elementary}(n)_a \ge (n/e)^a.\enq

\begin{prop}\label{Pconn}
With $p=e^2p'$, if 
$\, N(H,F) \leq \EE_{p'}X_F  \,$
for every connected $F$, then 
\beq{eq.MT'} 
N(H,F) \leq \EE_pX_F 
\enq
for every $F$.
\end{prop}

\nin
\emph{Proof.}
This follows from the next two assertions.
\begin{claim}\label{Cconn}
If $\, N(H,F) \leq \EE_{p'}X_F  \,$
for every connected $F$, then 
\eqref{eq.MT'} holds for every $F$ without isolated vertices.
\end{claim}
\begin{claim}\label{Ciso}
If \eqref{eq.MT'} holds for every $F$ without isolated vertices, then 
it holds for every $F$.
\end{claim}

\begin{subproof}[Proof of Claim~\ref{Cconn}.] 
Supposing $F$ is disconnected without isolates, let 
$F_1\dots F_s$ be the isomorphism types of components of $F$, with $m_i$ the number of components isomorphic to $F_i$, 
and set $v_i=v(F_i)$, $e_i=e(F_i)$ and $\aut_i=\aut(F_i)$.
Then

\[\begin{split}N(H,F) \le \prod_i \binom{N(H,F_{i})}{m_i} \le \prod_i \frac{N(H,F_{i})^{m_i}}{m_i!} &\le \prod_i \frac{(\EE_{p'}F_{i})^{m_i}}{m_i!}\\
&=\prod_i \frac{{\left[(n)_{v_i} (p')^{e_i}/\aut_i\right]}^{m_i}}{m_i!} 
\le \frac{(n)_{v_F}e^{v_F}(p')^{e_F}}{\aut(F)}\le\EE_pX_F.
\end{split}
\]
Here the fourth inequality uses 
$\aut(F)=\prod_i [\aut_i^{m_i}\cdot m_i!]\,$
and $\,\prod_i\left[(n)_{v_i}\right]^{m_i} \le n^{v_F} \le e^{v_F}(n)_{v_F}$ (see \eqref{elementary}), and the last
holds since $F$ isolate-free implies $v_F\leq 2e_F$.
\end{subproof}

\begin{subproof}[Proof of Claim~\ref{Ciso}.] 
Let $Y\neq \0$ be the set of isolated vertices of $F$, 
and $F'=F-Y$. 
Then, since $F'$ satisfies \eqref{eq.MT'},
    \[N(H,F)=N(H,F'){{v_H - v_{F'}} \choose |Y|}\le \EE_pX_{F'}{{n-v_{F'}} \choose |Y|}=\EE_pX_F.\qedhere\]
\end{subproof}

\mn

This completes the proof of Proposition~\ref{Pconn}.\qed

\section{Proof of \Cref{MT} for ``large" $q$}\label{sec.MT}

The rest of the paper is devoted to the proof of \Cref{MT}.
This section and the next
prove the first part of
the theorem (except for very small $q$; see \eqref{q'}), 
with the second part (a stronger statement for smaller $q$) treated in \Cref{sec.MT.sparse}.
So we are now given a $q$-sparse $H$ 
and 
should show that the inequality in
\eqref{eq.MT} holds for a given $F$, 
provided $p=K q\log^2n$ with 
a large enough (fixed) $K$.
We may assume $F$ is connected
(by Proposition~\ref{Pconn}), and
that 
$v_F \ge 1$ and $F \sub H$.
(We may of course also assume $p<1$, since 
otherwise \Cref{CpEpE*'} is trivial.)

\mn

We follow the sketch of \Cref{sec.prelim},
setting 
\beq{CC}
\CC = 16\cdot 20
\enq
(again, we do not try
for good constants),
\[
\mbox{$q':= \CC q\log n$, $\,p':=eq'\log n,\,$ and
$\,p=e^2p'$ ($=O(q\log^2 n))$,}
\]
and assuming 
\beq{q'}
q'\ge 1/n.
\enq
(This is the minor restriction on $q$ mentioned above,
which we use to allow application of 
\Cref{prop.small claims}(b,c)
in the proof of \Cref{F.gen.lem}.
The assumption could presumably be avoided, but there 
seems little point to this since 
the stronger result of \Cref{sec.MT.sparse} 
covers $q$ up to some $\gO(1/n).$)

Using \Cref{lem.leading seq}, we fix a sequence
\[
\emptyset = W_0 \subsetneq W_1 \subsetneq \cdots \subsetneq W_k=V(F)
\]
such that $[W_{i-1}, W_i]$ is $q'$-leading and 
$\mu_{q'}(W_{i-1}, W_i) \ge 1$ for each $i \in [k].$

Recall from Section~\ref{sec.prelim} that 
a \emph{labeled copy of $F$ in $H$} is an edge-preserving
injection from $V(F) $ to $ V(H)$.
Using $\tilde N(H,F)$ for the number of such copies,
we have
\beq{tilde.N.bd}  
N(H,F)\cdot \aut(F)=
\tilde N(H,F) \le \prod_{i=1}^k \max_U \tilde \tau_{[W_{i-1}, W_i]}(U,H),\enq
and will show that, for any $i \in [k]$ and $U \sub V(H)$,
\beq{tau.ub} \tilde \tau_{[W_{i-1},W_i]}(U,H) \le \tilde \mu_{p'}(W_{i-1},W_i).
\enq
This is enough for Theorem~\ref{MT}:  By \eqref{tau.ub}, the right-hand side of \eqref{tilde.N.bd} is at most
\[ \prod_{i=1}^k \tilde \mu_{p'}(W_{i-1},W_i) = 
\tilde \mu_{p'}(F)=\mu_{p'}(F)\cdot \aut(F),\]
so $N(H,F) \le \mu_{p'}(F)$. But,
using \eqref{elementary} (for the first inequality)
and $v_F\leq 2e_F$ (for the second), 
\[
\mu_{p'}(F)=\EE_{p'}X_F\cdot \frac{n^{v(F)}}{(n)_{v(F)}}\le \EE_{p'}X_F e^{v(F)} \le \EE_{p'}X_F e^{2e(F)} = \EE_pX_F;
\]
so we have \Cref{MT}.

\bn

In the rest of this section we prove \eqref{tau.ub}
modulo the proof of \Cref{lem.tree.hard}, which is 
postponed to \Cref{sec.lem.tree.hard}. 
Throughout the discussion we fix $i$,
set $A=W_i \setminus W_{i-1}$ and 
$B=N(A) \cap W_{i-1}$, and let $X$ be the subset of 
$U$ corresponding to $B$ (under the implicit orderings
of $U$ and $W_{i-1}$).
Then
\[
\tilde \tau_{[W_{i-1},W_i]}(U,H) \leq  \tilde \tau_{[B,A]}(X,H)=\tau_{[B,A]}(X,H) \cdot \aut(B,A).
\]
So, since 
$\tilde \mu_{p'}(W_{i-1}, W_i)=\tilde \mu_{p'}(B,A)=\mu_{p'}(B,A) \cdot \aut(B,A)$,
we will have \eqref{tau.ub} if we show
\beq{tau.ub'} 
\tau_{[B,A]}(X,H) \le \mu_{p'}(B,A).
\enq

\mn

This will follow easily from the next two lemmas.

\begin{lem}\label{F.fixed.lem1}
The maximum size of a collection of edge-disjoint copies of $[B,A]$ on $X$ in $H$ is at most
\[    
\max\{e\mu_{2q}(B,A), \log n\}.
\]    
\end{lem}

Define an \textit{$(X, [B,A])$-component} of $H$ to be a union of copies of $[B,A]$ on $X$ (in $H$) that cannot be expressed as an edge-disjoint union of two smaller such unions and is maximal with this property.
Since $H,B,A,$ and $X$ are now fixed, 
we henceforth use \emph{component} for 
``$(X, [B,A])$-component of $H$."

\begin{lem}\label{F.gen.lem}
If $e(B,A) \ne 0$, then no component
contains more than 
$(e\log n)^{e(B,A)-1}$ copies of $[B,A]$ on $X$.
\end{lem}

\begin{proof}[Proof of \eqref{tau.ub'}] We may assume $e(B,A) \ne 0$, since otherwise \eqref{tau.ub'} is trivial. 
Combining Lemmas~\ref{F.fixed.lem1} and Lemma~\ref{F.gen.lem}
(and the fact that the quantity in Lemma~\ref{F.fixed.lem1}
bounds the number of components) then gives
\begin{eqnarray*}
\tau_{[B,A]}(X,H) &\le &\max\{e\mu_{2q}(B,A), \log n\}\cdot (e\log n)^{e(B,A)-1}\\
&\le &\mu_{q'}(B,A)(e\log n)^{e(B,A)} 
\,\,=\,\,\mu_{p'}(B,A),
\end{eqnarray*}
where the second inequality uses 
$\mu_{q'}(B,A)= \mu_{q'}(W_{i-1},W_i)\ge 1$ 
(and, weakly, the definition of $q'$).
\end{proof}

\begin{proof}[Proof of \Cref{F.fixed.lem1}]
If there are $\nu$ edge-disjoint copies of $[B,A]$ on $X$ in $H$, with $R$ their union, then (using 
the trivial $e(B,A)\geq |B|$ for the 
final inequality)
    \[\begin{split}\EE_qX_R \le n^{|B|} \frac{\left(\mu_q(B,A)\right)^\nu}{\nu!} \le n^{|B|} \left(\frac{e\mu_q(B,A)}{\nu}\right)^\nu &=n^{|B|}2^{-\nu e(B,A)}\left(\frac{e\mu_{2q}(B,A)}{\nu}\right)^\nu\\
    &
\le (n2^{-\nu})^{|B|}\left(\frac{e\mu_{2q}(B,A)}{\nu}\right)^\nu.\end{split}\]
But $\EE_qX_R\geq 1$
(since $H$ is $q$-sparse and $R\sub H$), 
so we must have $\nu\le \max\{e\mu_{2q}(B,A), \log n\}.$
\end{proof}

For the proof of \Cref{F.gen.lem} we view components
as ``trees," as follows.
\nin
(Here $[W,Z]$ is general.)

\begin{mydef}\label{def.tree} 
A \textit{$[W,Z]$-tree} is $T=\cup_{i \ge 0} T_i$, 
where
\[     
\text{each $T_i$ is a copy of $[W,Z]$ on $W$, 
and $\emptyset \neq E(T_i) \cap (\cup_{j <i} E(T_j)) \ne E(T_i)$ for each $i \ge 1$.}
\]      
\end{mydef}
\nin
We also view this as a rooted graph, $T[W,V(T)]$,
and in what follows use $e_T$ for $e(W,V(T))$.
Notice that in our situation the components are
$[B,A]$-trees on $X$.

\begin{proof}[Proof of \Cref{F.gen.lem}] If $|A|=1$, 
then each component consists of a single copy of $[B,A]$,
and there is nothing to show; 
so we assume $|A| \ge 2$, and note that then 
\Cref{prop.small claims}(c) gives $e(B,A) \ge 3$
(here we use \eqref{q'} and 
$\mu_{q'}(B,A)=\mu_{q'}(W_{i-1},W_i)\geq 1)$.

Our first (and main) task is to give an upper bound 
(see \eqref{compsize}) on the sizes of the components;
that is, on $e_T$ for a $[B,A]$-tree $T$ 
that can appear 
in $H$.
For such a $T = \cup T_i$ (as in Definition~\ref{def.tree}), 
\beq{1EEq}
1\le \EE_qX_T \le n^{|B|} \mu_q(T[B,V(T)])
\enq
(the first inequality holding since $H$ is $q$-sparse); so
it is enough to show that for too large an $e_T$ 
the right-hand side of \eqref{1EEq} is less than 1.
This is mostly based on the next lemma,
which is proved in \Cref{sec.lem.tree.hard}.

\begin{lem}\label{lem.tree.hard}
Let $J$ be a graph on $Z$ and $W \subsetneq Z$. 
If $[W,Z]$ is $\rr$-leading and $J$ is $(\rr/2)$-sparse, 
then for any $[W,Z]$-tree $T$,
\beq{lem.tree.hard.stronger} 
\tilde \mu_\rr(T[W,V(T)]) \le \mu_\rr(W,Z) \cdot \min\{v(W,Z), 16\log n\}^{e_T}.
\enq   
\end{lem}
\nin
We believe that \eqref{lem.tree.hard.stronger} 
could be strengthened to 
\beq{improve3.4}
\mu_\rr(T[W,V(T)]) \le \mu_\rr(W,Z)K^{e_T}
\enq
for some constant $K$; 
as mentioned in Section~\ref{sec.prelim}, this
would remove a log factor in Theorem~\ref{MT}.
We also suspect that both Lemma~\ref{lem.tree.hard}
and the conjectured improvement
hold even without the sparseness 
assumption.

\mn

For now assuming \Cref{lem.tree.hard},
we return to bounding $e_T$ for the $T$ of \eqref{1EEq}, setting
\[
\gd=\min\{v(B,A), 16\log n\}/(\CC\log n)
\]
and 
\[
\gamma=2/\log (1/\delta) 
\,\,\,\, (\leq 2/\log (20); \,\text{see}\, \eqref{CC}).
\]

Noting that $[B,A]$ is $q'$-leading (since 
$[W_{i-1},W_i]$ is), and $F[A\cup B]$
is ($q'/2$)-sparse (since $H$ is $q$-sparse
and contains $F$),
we may apply \Cref{lem.tree.hard} to obtain
\[
\mu_{q'}(T[B,V(T)]) \,\,(\le \tilde \mu_{q'}(T[B,V(T)]))\,\,
\le \,\mu_{q'}(B,A)\cdot(\CC\gd\log n)^{e_T},
\]
whence the right-hand side of \eqref{1EEq} is 
\[
n^{|B|}\mu_{q'}(T[B,V(T)])\cdot (\CC\log n)^{-e_T}\le n^{|B|}\mu_{q'}(B,A)\delta^{e_T}.
\]

\nin 
This gives 
\beq{compsize}
e_T\le \gamma\cdot\max\{|B|,1\}\log n,
\enq
since otherwise $n^{|B|}\mu_{q'}(B,A)\gd^{e_T} 
< n \cdot n^{|B|}n^{-2\max\{|B|,1\}}\le 1,$
where the first inequality uses \Cref{prop.small claims}(a).

\mn

Using \eqref{compsize} (and, say, 
$e(B,A) \ge \max\{|B|,1\}$), we have the naive bound
\[\tau_{[B,A]}(X,T) \le {e_T \choose e(B,A)} \le (e\gamma\log n)^{e(B,A)},\]
which immediately gives \Cref{F.gen.lem} if 
(say) $e(B,A) >\log\log n$. 
To account for the extra $e\log n$ in case
\beq{small.eba} (3 \le )\,\,e(B,A)\le \log \log n, \enq
we may argue more carefully as follows. 

\mn

Here we are always talking about copies of $[B,A]$ 
\emph{on $X$ in $T$}, so will usually omit ``on $X$ in $T$."
We extend ``copy" to vertices and edges (or sets of edges)
in the natural way, and use overline for copies.

\mn

The combination of \eqref{small.eba} and 
Proposition~\ref{prop.small claims}(c) gives
$\gd = O(\log\log n/\log n)$ and
\beq{gamma.ub}
\gamma=O(1/\log\log n).
\enq
Recalling from \Cref{prop.small claims}(b) that vertices of $A$ have degree at least 2
(in $[B,A]$), fix $x \in A$ of minimum degree and some neighbor $y$ of $x$ in $A \cup B$,
and let $I=E([B,A]) \setminus \{xy\}$. 
We may specify a copy of $[B,A]$ 
via the following two steps.

\mn 
(a) Choose a copy $\cop{I}$ of $I$; the number of ways to do
this is (very crudely) at most 
$\C{e_T}{e(B,A)-1}.$

\mn
Notice that 
$\cop{I}$ determines either
(i) the pair $\{\cop{x}, \cop{y}\}$, 
or (ii) $\cop{x}$ itself, since:  if $y$ is another 
minimum degree vertex of $A$,
then $\cop{x}$ and $\cop{y}$ are the 
(only) two minimum degree vertices of $V(\cop{I})\sm X$;
and otherwise $\cop{x}$ is the unique such vertex.

\mn
(b) Choose $\cop{x}$ and $\cop{y}$. 
Here: if (i) above holds, then the number of 
possibilities is at most 2; and if (ii) holds,
then we know $\cop{x}$ and are choosing $\cop{y}$
from $X$ if $y\in B$,
and from $V(\cop{I}) \setminus (X \cup \{\cop{x}\})$
if $y\in A$, so the number of possibilities is at most
$
\max\{|A|, |B|\} \le e(B,A)
$
(see Proposition~\ref{prop.small claims}(c)).

\mn

In sum, the number of copies of $[B,A]$ (on $X$ in $T$) is at most
\[
{e_T \choose e(B,A)-1} e(B,A)  \le
\left\{\begin{array}{ll}
(e\gamma \log n)^{e(B,A)-1}\log\log n&\mbox{if $B\neq \0$},\\
(e\gamma \log n/2)^{e(B,A)-1}\log\log n&\mbox{if $B= \0$}
\end{array}\right.
\]
(where the inequality uses 
\eqref{compsize}, Proposition~\ref{prop.small claims}(c)
(when $B\neq \0$), and 
\eqref{small.eba}); 
and, in view of \eqref{gamma.ub} and the fact that 
$e(B,A) \ge 3$, this is less than the 
$(e\log n)^{e(B,A)-1}$ of
\Cref{F.gen.lem}.
\end{proof}

\section{Proof of \Cref{lem.tree.hard}}\label{sec.lem.tree.hard}

Let $T=\cup_{i \ge 0} T_i$ as in \Cref{def.tree}, and
set $Z_i=V(T_i)$,
$Y_i=Z_i \cap (\cup_{j<i}Z_i)$ 
($\supsetneq W$), $v_i=v(Y_i,Z_i)$, and
$e_i=e(Y_i,Z_i)$ (note this sacrifices edges of $T_i$
that are contained in $Y_i$ but not in any earlier $T_j$).    
Then
    \beq{mu.tilde.reform} 
    \begin{split} \tilde \mu_\rr(T[W,V(T)]) &=n^{v(W,T)}\rr^{e(W,T)} \le n^{v(W,Z)}\rr^{e(W,Z)} \prod_{i \ge 1} n^{v_i}\rr^{e_i} \\
    &\le \mu_\rr(W,Z)\cdot \aut(W,Z) \prod_{i \ge 1} \aut(Y_i,Z_i),\end{split}
    \enq
the last inequality holding since, for each $i \ge 1$,
\[
n^{v_i}\rr^{e_i}=\tilde \mu_\rr(Y_i, Z)=\mu_\rr(Y_i,Z_i)\cdot \aut(Y_i,Z_i) \,\le\, \aut(Y_i,Z_i).
\]
(Here $\mu_\rr(Y_i,Z_i)\le 1$
because $[W,Z]$ is $\rr$-leading, with equality possible
since we can have $Y_i=Z_i$.)

\mn

So, as suggested in 
\Cref{sec.prelim}, the main issue for \Cref{lem.tree.hard}
will be bounding
$\aut(Y,Z)$ for $W \sub  Y \subsetneq Z$. 
To this end, we introduce the following key notion.
(Note that, here and below, we speak of a 
 general $(W,Z)$,
not the $(W,Z)$ of \Cref{lem.tree.hard}.)

\begin{mydef}\label{def.sunflower}
A \emph{k-sunflower} in
$J[W,Z]$ is a disjoint collection 
$P_1, \ldots, P_k \sub Z \setminus W$ satisfying

\mn
{\rm (a)} $E(P_i,P_j)=\0$ $\,\,\forall i\neq j$, and

\mn
{\rm (b)} the rooted graphs $J[Z\sm \cup P_j,P_i]$ are all isomorphic (with respect to a fixed ordering of $Z\sm \cup P_j$).
\end{mydef}
\nin
In particular, $N_{J[Z]}(P_i) \setminus P_i$ takes the same value, say $Q$, for all $i$, and the 
$J[Q,P_i]$'s are all isomorphic.

To bound $\aut(W,Z)$, we first show, in 
\Cref{lem.no sunflower}, that the assumptions of \Cref{lem.tree.hard} rule out existence of a 
``large" sunflower.  Our main 
point,
\Cref{lem.sunflower}, then bounds $\aut(W,Z)$
in the absence of such a sunflower.
The derivation of \Cref{lem.tree.hard} is given
at the end of the section.

\begin{lem}\label{lem.no sunflower} If $J[W,Z]$ is 
$\rr$-leading and $J[Z]$ is $(\rr/2)$-sparse, then $J[W,Z]$ contains no $(\log n)$-sunflower.    
\end{lem}

\begin{proof} 
Let $k=\log n$, and suppose for a contradiction that
$J[W,Z]$ contains a $k$-sunflower $\{P_1, \ldots, P_k\}$,
with $Q$ the common value of $N_{J[Z]}(P_i)\sm P_i$.
Set $P_1=P$ (so $[Q,P_i]\cong [Q,P]$ $\forall i$), 
$S=\cup P_i$, 
and $R=J[Q \cup S]$. Then
\beq{eq.no sunflower}
\mu_{\rr/2}(R) \le n^{|Q|}\frac{\mu_{\rr/2}(Q, P)^k}{k!} =\frac{n^{|Q|}}{k!} \left[\mu_\rr(Q,P) 2^{-e(Q,P)}\right]^k.
\enq
But $\mu_\rr(Q, P) =\mu_\rr(Z\sm P,P)<1$
(since $N_J(P)=Q$ and $[W,Z]$ is $\rr$-leading);
so, since $e(Q, P) \geq |Q|$, 
the right-hand side of \eqref{eq.no sunflower} is less than
\[
\frac{n^{|Q|}}{k!}2^{-|Q|\log n} \ll 1,
\]
which 
is impossible since $J[Z]$ is $(\rr/2)$-sparse.
\end{proof}

\begin{lem}\label{lem.sunflower}
Let $k \ge 4 \log n$. If $J[W,Z]$ does not contain a 
$k$-sunflower and 
no vertex of $Z \setminus W$ is isolated in $J[W,Z]$,
then $\aut(W,Z)\le (16k)^{e(W,Z)}$.
\end{lem}

\nin
As noted in Section~\ref{sec.prelim}, 
we believe 
this is true without
the restriction on $k$ (maybe with 16 
replaced by some 
other constant), and that the question is of some 
interest (though it would not
immediately yield \eqref{improve3.4},
since the $\log n$ in Lemma~\ref{lem.no sunflower} 
is unavoidable).

\begin{proof}
Let $T_1, \dots, T_m$ be the 
(connected) components of $J[Z \setminus W]$,
and say $T_i$ and $T_j$ are \emph{equivalent} if 
$J[W,T_i]$ and $ J[W,T_j]$ are isomorphic.
The set of $T_i$'s breaks into equivalence classes
$\T_1,\dots ,\T_{a}$, the members of each $\T_j$
forming a sunflower, say of size
$k_j\,$ ($< k$), and we have
\beq{aut.prod.ub}
\aut(W,Z) \le 
\prod_j k_j! \prod_{T_i \in \T_j} \aut(W,V(T_i))
\le
\prod_j k_j! \prod_{T_i \in \T_j} \aut(T_i).
\enq

\begin{claim}\label{cl.aut.i} For each $i$,
\[ 
\aut(T_i)\le 
\min\{(4k)^{e(T_i)},(2e(T_i))^{2e(T_i)}\}.
\]
\end{claim}
\nin
(The second bound is trivial but helpful below.)

\mn

Before proving Claim~\ref{cl.aut.i} we show that it gives 
Lemma~\ref{lem.sunflower}.
Denote by $\Tjs$ the common isomorphism type 
of the $T_i$'s in $\T_j$, and set 
$\vjs =v(\Tjs)$ and $\ejs =e(\Tjs)$.
The right-hand side of \eqref{aut.prod.ub} is then
\[
\prod_j k_j! (\aut(\Tjs))^{k_j}
~\leq ~
\prod_j [k_j\cdot \aut(\Tjs)]^{k_j},
\]
and we want to show this is at most 
$(16k)^{e(W,Z)} \ge \prod_j(16k)^{k_j\ejs}$. 
So it is enough to show
\[
k_j\cdot \aut(\Tjs)\leq (16k)^{\ejs},
\]
which in view of Claim~\ref{cl.aut.i}
will follow from
\beq{L5.3to.show}
k_j^{1/\ejs}\cdot \min\{4k,(2\ejs)^2\}
\leq 16k.
\enq

Here we have plenty of room.
If $k_j \leq 4^{\ejs}$ then the bound is immediate
(note $e_j^*\ne 0$, since the $T_i$'s are not
singletons);
and otherwise
$k \geq k_j\geq 4^{\ejs}$, and
\eqref{L5.3to.show} follows from the generous
\[
16k^{1-1/\ejs} \ge 16\cdot 4^{\ejs-1}\geq (2\ejs)^2.
\qedhere\]
\end{proof}

\begin{proof} [Proof of Claim~\ref{cl.aut.i}]

To slightly simplify notation, we now
write $\ST$ for $T_i$.
Since $v_\ST \ge 2$ (not actually needed here since 
the assertion is trivial when $v_R=1$)
the claim's second bound follows from 
$\aut(R)\leq v_R^{v_R}$ and $v_R\le 2e_R$.
So we are interested in the first bound,
for which we will need the trivial observation that,
for any connected $G$,
\beq{autG}
\aut(G) \leq v_G \gD_G^{v_G-1}
\enq
(where $\gD$ is maximum degree).

Set 
$L = \{v \in V(\ST): d_\ST(v) > \log v_\ST\}$.
We first note that if
$L=\0$ then the desired $\aut(R) < (4k)^{e_R}$ is 
easy:  we have $\gD_R\leq \log v_R$ and $e_R\geq v_R-1$
(since $R$ is connected), so by \eqref{autG},
\[
\aut(R)\leq v_R(\log v_R)^{v_R-1} \leq v_R(\log v_R)^{e_R} 
< v_R k^{e_R} \leq (4k)^{e_R}
\]
(where the last inequality uses $x\leq 4^{x-1}$ for $x\geq 1$).
So we assume $L\neq \0$.

Let 
the components of $\ST-L$ be $S_1,\ldots, S_t$,
and partition these into
equivalence classes according to the isomorphism type of 
$[W \cup L, S_i]$;
say the classes are $\sss_1,\dots ,\sss_{b}$,
noting that 
$\ell_j:=|\sss_j|< k$
since $\sss_j$ is a sunflower.
As earlier we use $\Sjs$ for the isomorphism type of the $S_i$'s 
in $\sss_j$
and now set $\vjs =v(\Sjs)$ and $\ejs =e(\Sjs)$.

\mn

Since $|L| < 2e_\ST/\log v_\ST$, we have
\[
\aut(\ST[L])\le |L|! \le |L|^{|L|} \leq 
v_\ST^{2e_\ST/\log v_\ST} = 4^{e_\ST},
\]
whence
\beq{autRL}
\aut(\ST)\le \aut(\ST[L])\prod_j\ell_j!\prod_{S_i \in \sss_j}\aut(S_i)
\le 4^{e_\ST}
\prod_j\left[\ell_j\cdot\aut(\Sjs)\right]^{\ell_j}.
\enq
So we should 
show that the right-hand side of \eqref{autRL} is at most
$(4k)^{e_R}$,
for which,
since $e_R=\sum \ell_j\ejs$, it is enough to show
$\ell_j\cdot\aut(\Sjs) \leq k^{\ejs}$  for each $ j$.
In fact, noting that $\ejs\geq \vjs$ 
(since each $S_i$ is connected with $E(S_i,L)\neq \0$),
we show the stronger
\[
\ell_j\cdot\aut(\Sjs) \leq k^{\vjs}
\,\,\,\,\forall j.
\]

As was true for \eqref{L5.3to.show}, the argument here
depends on $\ell_j$.
If $\ell_j \leq 2^{\vjs}$, then \eqref{autG},
with $\gD_{\Sjs} \le \log v_R$, gives
\[
\ell_j\cdot\aut(\Sjs) 
\,\le 2^{\vjs} \vjs(\log v_\ST)^{\vjs-1} 
< (4 \log v_\ST)^{\vjs} < k^{\vjs}.
\]
And if
$\ell_j> 2^{\vjs}$, we have the even easier
\[
\ell_j\cdot\aut(\Sjs) 
\le 
\ell_j (\vjs)^{\vjs} \leq \ell_j^{\vjs} < k^{\vjs},
\]
where the second inequality follows from 
$\ell_j> 2^{\vjs}$ and $2^{\vjs-1}\geq \vjs$.
\end{proof}

Finally, to complete the proof of \Cref{lem.tree.hard}, we 
observe that under its assumptions,
Lemmas~\ref{lem.no sunflower} and \ref{lem.sunflower} imply 
the following bound, which we may apply to each 
$\aut(\cdot)$ factor on the right-hand side of \eqref{mu.tilde.reform}.

\begin{cor}\label{lem.aut.ub} If $J[W,Z]$ is 
$\rr$-leading and $J[Z]$ is $(\rr/2)$-sparse, 
then
\[
\aut(W,Z)\le 
\min\{v(W,Z),16\log n\}^{e(W,Z)}.
\]
\end{cor}
\nin
(When $W=Z$, the bound is $0^0=1$.)

\begin{proof}
We may assume  
$v(W,Z) \ge 2$ (or the assertion is trivial), and
note that $r$-leading then implies that $Z\sm W$ 
contains no isolated vertices.
The first bound
is thus contained in the following more general
statement, which mostly ignores the corollary's hypotheses.
\begin{prop}
If $J[W,Z]$ has no isolated vertices in $Z\sm W$, 
then $\aut(W,Z)\leq v(W,Z)^{e(W,Z)}$.
\end{prop}

\begin{subproof}[Proof \emph{(briefly).}] 
We work in $G:=J[Z\sm W]$, with $x$ always in $V(G)$,
and use $d$ for degree in $G$, and $v$ for $v_G$ ($=v(W,Z)$). 

We will show the stronger
\beq{aut.eas}\aut(G) \leq v^{e_G}.\enq 
It suffices to prove \eqref{aut.eas} 
assuming the  components of $G$ are all isomorphic,
with the common isomorphism type a tree 
(since otherwise $v \leq e_G$ and \eqref{aut.eas} follows from 
the trivial $\aut(G) \leq v^v$).

If $G$ is a matching, then $\aut(G) = (v/2)!2^{v/2} < v^{e_G}$. 
Otherwise, $L:=\{x: d(x) = 1\}$ is independent,
so $|L|\le e_G$; which again gives \eqref{aut.eas}, since any 
$\vp\in \Aut(G)$ is determined by $\vp|_L$, and the number of possibilities for $\vp|_L$ is at most
$v^{|L|}$.
\end{subproof}

For the corollary's second bound we invoke
Lemma~\ref{lem.no sunflower} to say
$J[W,Z]$ contains no $(\log n)$-sunflower,
and apply
Lemma~\ref{lem.sunflower}. 
\end{proof}

\section{Proof of \Cref{MT} for ``small" $q$}\label{sec.MT.sparse}

Let $p = \gb/n$ and $q =\alpha p =:\gc/n<1/(3n)$, 
with $\ga\in (0,1)$ to be specified.
Recalling that we assume $F$ is connected
(as justified by Proposition~\ref{Pconn}), we first observe that we  may in fact assume $F$ is a tree; this just 
requires $q<1/n$:
for a cycle $C$ of length $t$,
\[
\EE_qX_C < (nq)^t/(2t) <1,
\]
so $C$ cannot be contained in a $q$-sparse $H$.

Set $v_F=m~(=e_F+1)$, noting that 
\[
\EE_pX_F=(n)_mp^{m-1}/\aut(F)=(n)_m(q/\ga)^{m-1}/\aut(F).
\]
Write $\nu(G,J)$ for the largest size of a 
vertex-disjoint collection of copies of $J$ in a graph $G$,
and notice that
\beq{NHF}
\begin{split}N(H,F)&=\sum_{t \ge m} \sum\{N(T,F):\text{$T$ is a $t$-vertex component of $H$}\}\\
&\le \sum_{t \ge m}\sum_{v_T=t}\nu(H,T)N(T,F),\end{split}
\enq
where in the second line
$T$ runs over (isomorphism types of) trees.
\begin{claim}\label{Cnu}
For any $T$,
$
\,\,\,\nu(H,T) \le e\cdot \EE_qX_T.
$
\end{claim}

\begin{proof}
Let $R$ consist of $\nu=\nu(H,T)$ vertex-disjoint copies of $T$ in $H$, where $v_T=t$. Then, since $H$ is $q$-sparse,
\[
1\le \EE_q X_R =\frac{(n)_{v_R} q^{e_R}}{\aut(R)}=\frac{(n)_{\nu t} q^{\nu(t-1)}}{(\aut(T))^\nu\cdot \nu!} \le \frac{(\EE_qX_T)^\nu}{\nu!} \le \left(\frac{e\cdot \EE_q X_T}{\nu}\right)^\nu
\]
(and the claim follows).\end{proof}

Inserting Claim~\ref{Cnu} in \eqref{NHF}, we have
\[
N(H,F)<e\sum_{t \ge m} \sum_{v_T=t} \EE_qX_T\cdot N(T,F),
\]
in which, crucially, the inner sum is
\beq{sumvT}
\EE|\{(T,F):v_T=t, F \sub T \sub G_{n,q}\}|,
\enq
where (here and below) we
use $T$ and $F$ for \emph{copies of} $T$ and $F$.

For the expectation in \eqref{sumvT}
we need the number
of $(T,F)$'s in $K_n$.  This is 
$N(K_n,F)= (n)_m/\aut(F)$ times the number of ways to extend
a given $F$ to some $T$ of size $t$;
and the latter number is 
${n-m \choose t-m}$ times the number of labeled, 
$m$-component, rooted forests on $[t]$ 
with specified roots, which is known to be 
$mt^{t-m-1}$ (see \cite{takacs1990cayley}). 
Thus \eqref{sumvT} is 
\[\begin{split}
(n)_m/\aut(F) \cdot {n-m \choose t-m}mt^{t-m-1}\cdot q^{t-1}
&< \frac{n^m}{\aut(F)} \left(\frac{en}{t-m}\right)^{t-m}t^{t-m} q^{t-1}\\
&\le en(e\gc)^{t-1}/\aut(F),\end{split}\]
where the second inequality uses $\left(\frac{t}{t-m}\right)^{t-m} \le e^m$.
So for small enough $\ga$,
\begin{eqnarray*}
N(H,F) &<&e^2\sum_{t \ge m} n(e\gc)^{t-1}/\aut(F) 
\,\,=\,\,
e^2n \cdot (e\gc)^{m-1}/[(1-e\gc)\aut(F)]\\
&<& (n)_m(\gb/n)^{m-1}/\aut(F) 
\,\,=\,\,\EE_pX_F.
\end{eqnarray*}
(We need roughly $(\gc/\gb=$) $\ga< e^{-2}$ to allow for the 
$e^m$ in the first line and the possibility 
that $(n)_m$ is like $(n/e)^m$,
while $\gc<1/3$ makes the $(1-e\gc)$ in the first line
irrelevant.) 
\qed

\section*{Acknowledgments}
QD was supported by NSF Grants DMS-1954035 and DMS-1928930. JK was supported by NSF Grants DMS-1954035 and DMS-2452069.  
JP was supported by NSF Grant DMS-2324978, NSF CAREER Grant DMS-2443706 and a Sloan Fellowship.

\bibliographystyle{plain}
\bibliography{bibliography}

\begin{thebibliography}{10}

\bibitem{bollobas1987threshold}
B{\'e}la Bollob{\'a}s and Andrew Thomason.
\newblock Threshold functions.
\newblock {\em Combinatorica}, 7(1):35--38, 1987.

\bibitem{DKP3}
Quentin Dubroff, Jeff Kahn, and Jinyoung Park.
\newblock On the second {K}ahn-{K}alai conjecture: cliques, cycles, and trees.
\newblock {\em in preparation}.

\bibitem{erdos1960evolution}
Paul Erd\H{o}s and Alfr{\'e}d R{\'e}nyi.
\newblock On the evolution of random graphs.
\newblock {\em Publ. Math. Inst. Hungar. Acad. Sci}, 5:17--61, 1960.

\bibitem{frankston2021thresholds}
Keith Frankston, Jeff Kahn, Bhargav Narayanan, and Jinyoung Park.
\newblock Thresholds versus fractional expectation-thresholds.
\newblock {\em Annals of Mathematics}, 194(2):475--495, 2021.

\bibitem{kahn2007thresholds}
Jeff Kahn and Gil Kalai.
\newblock Thresholds and expectation thresholds.
\newblock {\em Combinatorics, Probability and Computing}, 16(3):495--502, 2007.

\bibitem{KLLM}
Peter Keevash, Noam Lifshitz, Eoin Long, and Dor Minzer.
\newblock Hypercontractivity for global functions and sharp thresholds.
\newblock {\em arXiv preprint arXiv:1906.05568}, 2019.

\bibitem{mossel2022second}
Elchanan Mossel, Jonathan Niles-Weed, Nike Sun, and Ilias Zadik.
\newblock On the second \mbox{Kahn--Kalai} conjecture.
\newblock {\em arXiv preprint arXiv:2209.03326}, 2022.

\bibitem{park2024proof}
Jinyoung Park and Huy~Tuan Pham.
\newblock A proof of the \mbox{Kahn--Kalai Conjecture}.
\newblock {\em Journal of the American Mathematical Society}, 37(1):235--243,
  2024.

\bibitem{takacs1990cayley}
Lajos Tak{\'a}cs.
\newblock On \mbox{Cayley's} formula for counting forests.
\newblock {\em Journal of Combinatorial Theory, Series A}, 53(2):321--323,
  1990.

\bibitem{talagrand1995all}
Michel Talagrand.
\newblock Are all sets of positive measure essentially convex?
\newblock In {\em Geometric Aspects of Functional Analysis: Israel Seminar
  (GAFA) 1992--94}, pages 295--310. Springer, 1995.

\bibitem{talagrand2010many}
Michel Talagrand.
\newblock Are many small sets explicitly small?
\newblock In {\em Proceedings of the forty-second ACM symposium on Theory of
  computing}, pages 13--36, 2010.

\end{thebibliography}

\end{document}